\documentclass[11pt]{article}
\usepackage[T1]{fontenc} 
\usepackage{amsmath} 
\usepackage{amsthm}
\usepackage{amssymb}
\usepackage{mathtools}

\usepackage[percent]{overpic}
\usepackage{tikz}
\usepackage{xcolor}

\usepackage{enumitem}
\usepackage{hyperref}
\hypersetup{
	colorlinks   = true, 
	urlcolor     = blue, 
	linkcolor    = blue, 
	citecolor   = red 
}
\usepackage[final
]{showlabels}

\setenumerate[2]{label=(\alph*)}
\setenumerate[1]{label=\textup{(\roman*)}}
\setenumerate[3]{label=(\arabic*)}

\newcommand{\normW}[1]{{\vert\kern-0.25ex\vert\kern-0.25ex\vert #1 
		\vert\kern-0.25ex\vert\kern-0.25ex\vert}}
\newcommand{\normWbig}[1]{{\Big\vert\kern-0.25ex\Big\vert\kern-0.25ex\Big\vert #1 
		\Big\vert\kern-0.25ex\Big\vert\kern-0.25ex\Big\vert}}

\newcommand{\C}{\mathbb{C}}
\newcommand{\N}{\mathbb{N}}

\newcommand{\R}{\mathbb{R}}

\newcommand{\boC}{\mathcal{C}}

\renewcommand{\Re}{\operatorname{Re}}
\renewcommand{\Im}{\operatorname{Im}}

\theoremstyle{plain}
\newtheorem{theorem}{Theorem}[section]

\newtheorem{corollary}[theorem]{Corollary}

\theoremstyle{definition}

\newtheorem{lemma}{Lemma}

\newtheorem{definition}[theorem]{Definition}

\newtheorem{remark}[theorem]{Remark}

\theoremstyle{remark}

\newtheorem*{merci}{Acknowledgments}

\numberwithin{equation}{section}
\begin{document}
\renewcommand{\thefootnote}{\fnsymbol{footnote}}
\title{Stability and instability of the quasilinear Gross--Pitaevskii dark solitons}

\author{Erwan Le Quiniou\footnotemark[1]}
\footnotetext[1]{\noindent Univ.\ Lille, CNRS, Inria, UMR 8524 - Laboratoire Paul Painlev\'e, F-59000 Lille, France.\\ E-mail: erwan.lequiniou@univ-lille.fr}
\date{}
\maketitle

%
%
\begin{abstract} 
	We study a quasilinear Schr\"odinger equation with nonzero conditions at infinity. In previous works, we obtained a continuous branch of traveling waves, given by dark solitons indexed by their speed. Neglecting the quasilinear term, one recovers the Gross--Pitaevskii equation, for which the branch of dark solitons is stable.
	Moreover, Z.~Lin showed that the Vakhitov--Kolokolov~(VK) stability criterion (in terms of the momentum of solitons) holds for general semilinear equations with nonvanishing conditions at infinity. 
	
	In the quasilinear case, we prove that the VK stability criterion still applies, by generalizing Lin's arguments. Therefore, we deduce that the branch of dark solitons is stable for weak quasilinear interactions. For stronger quasilinear interactions, a cusp appears in the energy-momentum diagram, implying the stability of fast waves and the instability of slow waves.  
\end{abstract}

\medskip
\noindent{{\em Keywords:}
	Quasilinear Schr\"odinger equation, Gross--Pitaevskii equation, traveling waves, dark solitons, nonzero conditions at infinity, orbital stability.
	
	\medskip
	\noindent{2020 {\em MSC}}:
	35Q55, 
	35C07, 
	35C08, 
	35J62,  
	35B35,  
	34L40, 
	34L05, 
	35Q60, 
	82D50, 
	34B40.  	
	\medskip
	
	\section{Introduction}
	We consider the following quasilinear Gross--Pitaevskii equation in one dimension
	\begin{equation}	\label{QGP}
		i\partial_t\Psi+\partial_{xx}\Psi+\Psi(1-|\Psi|^2)+\kappa\Psi\partial_{xx}(1-|\Psi|^2)=0,\quad\text{ in }\R\times\R,
	\end{equation}
	where $\kappa\in\R$ is a parameter measuring the strength of the quasilinear term, and $\Psi:\R\times\R\to\C$, satisfies the nonzero conditions at infinity
	\begin{equation}
		\label{nonzero0}
		\lim_{|x|\to \infty}| \Psi(x,\cdot)|=1.
	\end{equation} 
	Equation \eqref{QGP} was introduced as a generalization of the case $\kappa=0$ also known as the Gross-Pitaevskii equation, which is a standard model in nonlinear optics and quantum hydrodynamics \cite{kivshar,fibich,kevFrCa0}. For instance, in the context of optical fibers in \cite{krolikowski2000}, they obtained \eqref{QGP} with $\kappa\geq0$  to model weak nonlocal interactions in the refracting index of the beam of light. We refer to \cite{HorikisWnonloc}, for other relevant weakly nonlocal models. Equation \eqref{QGP} was also obtained with $\kappa\leq0$ in \cite{kurihara} to describe superfluid films with surface tension.

	Equation \eqref{QGP} is a particular case of the following quasilinear Schr\"odinger equation in $\R^d$
	\begin{equation}\label{eq:QLS}
		i\partial_t\Psi+\Delta\Psi+\Psi F(|\Psi|^2)-\kappa\Psi h'(|\Psi|^2)\Delta h(|\Psi|^2)=0,\quad\text{ in }\R^d\times\R,
	\end{equation}
	with $F(y)=1-y$ and $h(y)=1-y$, $\kappa\in\R$. To our knowledge, \eqref{eq:QLS} was only considered with vanishing conditions at infinity by the mathematical community. 
	Studying \eqref{eq:QLS} with $d=1$,  I.~D.~Iliev and K.~P.~Kirchev~\cite{ilievquasilinstab} obtained a continuous branch of solitary-wave solutions of the form
	$\Psi(x,t)=u_\omega(x)e^{i\omega t},$ and showed that the wave is stable whenever $\partial_\omega||u_\omega||_{ L^2(\R)}^2>0$ and unstable when $\partial_\omega||u_\omega||_{ L^2(\R)}^2<0$.
	In arbitrary space dimension, M.~Colin, L.~Jeanjean, and M.~Squassina~\cite{Colin2} obtained the existence and stability or instability by blowup of the ground state solitary-wave for \eqref{eq:QLS} with power law nonlinearities.
	In the semilinear regime ($\kappa=0$ in \eqref{eq:QLS}), we recover
	\begin{equation}\label{eq:GP}
		i\partial_t\Psi+\Delta\Psi+F(|\Psi|^2)\Psi=0,\quad\text{ in }\R^d\times\R.
	\end{equation}
	With vanishing conditions at infinity, M.~I.~Weinstein showed in~\cite{weinstein} the stability of the solitary wave solution to \eqref{eq:GP} satisfying a ground state property using Lyapunov functionals.
	Assuming $F(\rho_0)=0$ and that $\Psi$ satisfies the nonzero conditions at infinity
	\begin{equation}\label{eq:nonzero2}
		\lim_{|x|\to\infty}|\Psi(x,t)|^2=\rho_0, \quad\text{ for all }t\in\R,
	\end{equation}
	Z.~Lin investigated in \cite{linbubbles} in one dimension of space, the stability of nonvanishing traveling-wave solutions to \eqref{eq:GP} of the form,
	\begin{equation}\label{Twans}
		\Psi_{c}(x,t)=u_{c}(x-ct).
	\end{equation}
	In particular, he rigorously proves the VK criterion~\cite{vakhitovcrit} using the framework of abstract Hamiltonian systems presented in the celebrated work of  M.~Grillakis, J.~Shatah, and W.~Strauss \cite{grillakisshatah}.
	We also refer to the work of D.~Chiron in \cite{chiron-stability} which constitutes, to our knowledge, the most complete description of the existing orbital stability results for the nontrivial traveling-waves solutions of equation \eqref{eq:GP} and in particular deals with solutions vanishing at some point. In this work, we explain how the ideas of Lin can be generalized to investigate the stability of traveling-wave solutions to the quasilinear equation \eqref{QGP}. As in \cite{linbubbles}, we will perform the analysis on the hydrodynamical unknowns $\eta, v=\partial_x\theta$ with $\Psi=\sqrt{1-\eta}e^{i\theta}$, so that equation \eqref{QGP} formally yields
	\begin{equation}\label{eq:hydrot}
		\left\{\begin{aligned}
			&\eta_t=\partial_x(2v(1-\eta)),\\
			&v_t=-\partial_x\left(\frac{\eta_{xx}}{2(1-\eta)}+\frac{(\eta_x)^2}{4(1-\eta)^2}+v^2-\eta-\kappa\eta_{xx}\right).
		\end{aligned}\right.
	\end{equation}
	As explained in \cite{delaire2023exotic}, the system \eqref{eq:hydrot} is a particular case of the Euler-Korteweg system, for which the local well posedness was investigated by S. Benzoni-Gavage, R. Danchin and S. Descombes \cite{benzoniLWP} in the space $ H^{s}(\R)\times H^{{s-1}}(\R)$ with $s>2+1/2$. In \cite{benzoniStab} with D. Jamet they showed existence of homoclinic traveling waves and proved orbital stability of these waves when the VK stability criterion is satisfied. In the same setting, C. Audiard \cite{audiard2017} generalized the ideas in \cite{linbubbles} to prove the VK instability criterion. The present work will detail their proof of the VK stability and instability criterion in the particular case of \eqref{QGP}.
	Finally, after this paper was completed, the author found that in \cite{walshStab} K. Varholm, E. Wahl\'en and S. Walsh construct a general framework more suited than \cite{grillakisshatah} for the stability of solutions to hydrodynamical systems such as \eqref{eq:hydrot}.

	The quasilinear Gross--Pitaevskii equation \eqref{QGP} is Hamiltonian and formally preserves the energy
	\begin{equation}\label{eq:energie}
		E_\kappa(\Psi(\cdot,t))= \int_\R |\partial_x\Psi(x,t)|^2 dx +\frac{1}{2}\int_\R
		\left(1-|\Psi(x,t)|^2\right)^2dx-\frac{\kappa}{2}\int_\R\left (\partial_x |\Psi(x,t)| ^2 \right )^2  dx,
	\end{equation}
	as well as the (renormalized) momentum, whenever $\inf_{x\in\R}|\Psi(x,t)|>0$
	\begin{equation}\label{def:moment}
		P(\Psi(\cdot,t))=-\int_\R\Re( i\partial_x\Psi(x,t)\overline{\Psi(x,t)})\left(1-\frac{1}{|\Psi(x,t)|^2}\right)dx.
	\end{equation}
	Because of the nonzero conditions at infinity \eqref{nonzero0}, the natural energy space to study \eqref{QGP} is (see \cite{delaire2023exotic})  
	\begin{equation}\label{def:boX}
		\mathcal{X}(\R)=\{u\in  H^1_{\rm loc}(\R;\C):u'\in  L^2(\R),1-|u|^2\in L^2(\R)\}.
	\end{equation} 
	We endow $\mathcal{X}(\R)$ with two metrics, the natural energy distance, which writes, for $\psi$, $\tilde{\psi}\in\mathcal{X}(\R)$
	\begin{equation}\label{def:metric}
		d_{\mathcal{X}}(\psi,\tilde{\psi})=||\partial_x\psi-\partial_x\tilde{\psi}||_{ L^2(\R)}+|||\psi|-|\tilde{\psi}|||_{ L^2(\R)}+|\psi(0)-\tilde{\psi}(0)|,
	\end{equation} 
	and the hydrodynamical distance defined for nonvanishing functions in $\mathcal{X}(\R)$ i.e.\ for $\psi=\sqrt{1-\eta}e^{i\theta}$ and $\tilde{\psi}=\sqrt{1-\tilde\eta}e^{i\tilde\theta}$ in $\mathcal{X}(\R)$ with $\inf_\R|\psi|>0$ and $\inf_\R|\tilde{\psi}|>0$, we define
	\begin{equation}\label{def:metric2}
		d_{\mathrm{hy}}(\psi,\tilde{\psi})=||\eta-\tilde{\eta}||_{ H^1(\R)}+||\partial_x\theta-\partial_x\tilde\theta||_{ L^2(\R)}+\left |\mathrm{Arg}\left(\frac{\psi(0)}{\tilde{\psi}(0)}\right)\right|,
	\end{equation}
	where $\mathrm{Arg}$ is a continuous determination of the argument of complex numbers defined near $1$.
	In a previous work with A.~de Laire~\cite{delaire2023exotic}, we were interested in the classification and stability properties of the so-called finite energy traveling-wave solutions to \eqref{QGP}. Namely, studying solutions to \eqref{QGP} of the form \eqref{Twans}
	which represents a traveling wave with profile $u_{c}=u_{c,\kappa}\in\mathcal{X}(\R)$ propagating at speed $c\in\R$. Hence, the profile $u_{c,\kappa}$ satisfies
	\begin{equation}
		\label{TWc}\tag{TW$(c,\kappa)$}
		-icu_{c,\kappa}'+u_{c,\kappa}''+u_{c,\kappa} (1-|u_{c,\kappa}|^2)-\kappa u_{c,\kappa}\left(|u_{c,\kappa}|^2\right)''=0.
	\end{equation}
	Taking the complex conjugate of equation~\eqref{TWc}, we can assume without loss of generality that $c\geq0$. We obtained a continuous branch of $\boC^{2}(\R)$-solutions for every $(c,\kappa)\in\mathcal{D}=[0,\sqrt{2})\times(-\infty,1/2)$. A (partial) version of our result writes
	\begin{theorem}[Theorem~1, Proposition~3.8 in \cite{delaire2023exotic}]
		\label{thm:classiftw}
		Let $\kappa\in(-\infty,1/2)$ and $c\geq0$.
		\begin{enumerate}
			\item If $c\geq\sqrt{2}$ and $u_{c,\kappa}\in\mathcal{X}(\R)\cap\boC^{2}(\R)$  satisfies \eqref{TWc}, then $u_{c,\kappa}$ is trivial, i.e.\ there exists $\varphi\in\R$ such that 
			$$u_{c,\kappa}(x)=e^{i\varphi},\quad\text{ for all } x\in\R.$$  
			\item If $0\leq c<\sqrt{2}$ so that $(c,\kappa)\in\mathcal{D}$, then there exists a nonconstant function $u_{c,\kappa}\in\mathcal{X}(\R)\cap\boC^{2}(\R)$ called dark soliton that satisfies \eqref{TWc}.
			Moreover, this solution is unique up to translation and phase change, i.e.\
			any other solution $\phi\in\mathcal{X}(\R)\cap\boC^{2}(\R)$ satisfies for some $(x_0,\varphi)\in\R\times\R$ $$\phi(x)=u_{c,\kappa}(x-x_0)e^{i\varphi},\quad\text{ for all }x\in\R.$$ 
			Additionally, the function $(c,\kappa,x)\mapsto u_{c,\kappa}(x)$ is smooth in $\mathcal{D}\times\R$ (where the derivative at $c=0$ is in the sense of Dini) and satisfies the following decay estimate: for every multi-index $\alpha=(\alpha_1,\alpha_2,\alpha_3)\in\N^3$ and $D^\alpha=\partial_c^{\alpha_1}\partial_\kappa^{\alpha_2}\partial_{x}^{\alpha_3}$, there exist $A,C>0$ continuously depending on $(c,\kappa)$ such that
			\begin{equation}\label{eq:decaydark}
				|D^\alpha\partial_xu_{c,\kappa}|+|D^\alpha(1-|u|^2)|\leq Ae^{-C|x|},\quad\text{ for all }x\in\R.
			\end{equation}
			Also $|u_{c,\kappa}|$ is even with $|u_{c,\kappa}|'>0$ in $(0,\infty)$. Finally we have the bounds $1>|u_{c,\kappa}(x)|\geq c/\sqrt{2}$ for all $x\in\R$ with equality at $x=0$, so that $\inf_\R|u_{c,\kappa}|>0$ if and only if $c>0$.
		\end{enumerate} 
	\end{theorem}
	In the case $c=0$, the traveling-wave $u_{0,\kappa}$ vanishes at $x=0$ and $u_{0,\kappa}$ is also called black soliton or kink soliton \cite{chironexistence1d,bethuel2008existence}. Showing the stability of the black soliton requires other techniques \cite{bethuel-black,chiron-stability} since defining a momentum in that case requires to work modulo $\pi/2$ (see Appendix C in \cite{deLGrSm1}). Theorem~1 in \cite{delaire2023exotic} also show the existence of supersonic solution with $(c,\kappa)\in(\sqrt{2},\infty)\times(1/2,\infty)$, but in that case, the
	linearized Hamiltonian operator around the traveling wave has essential spectrum in $(-\infty,\lambda)$ for some $\lambda>0$, thus the analysis presented in this work, which is based on~\cite{linbubbles}, is not suited to study the stability of these waves. 
	In its weak formulation, equation \eqref{TWc} can be recast as 
	$$\begin{aligned}
		c\frac{d}{d s}P(u+sh)\rvert_{s=0}=\frac{d}{d s}E_\kappa(u+sh)\rvert_{s=0},
	\end{aligned}$$
	which is the Euler--Lagrange equation associated with the minimization problem
	\begin{equation}\label{Emin}
		\mathcal{E}_\kappa(\mathfrak{q})=\inf\{E_\kappa(u):u\in\mathcal{X}(\R),\inf_{x\in\R}|u(x)|>0,P(u)=\mathfrak{q}\}.
	\end{equation} 
	When $\kappa\leq0$ so that \eqref{Emin} defines nonnegative function, we showed stability and a ground state property for the dark soliton using a concentration-compactness argument, generalizing the approach used for the Gross--Pitaevski equation (equation \eqref{QGP} with $\kappa=0$) by F.~Bethuel P.~Gravejat and J.~C.~Saut in~\cite{bethuel2008existence} (see also \cite{delaire-mennuni}).
	\begin{theorem}[Theorem~1.10, 1.12 and Proposition~6.2 in \cite{delaire2023exotic}]\label{thm:min}
		Let $\kappa\in\R$ and $\mathcal{E}_\kappa$ be given by \eqref{Emin}, then the following statements hold.
		\begin{enumerate}
			\item If $\kappa>0$, then $\mathcal{E}_\kappa(\mathfrak{q})=-\infty,$ for all $\mathfrak{q}\in\R.$
			\item If $\kappa\leq0$, then there exist constants $\mathfrak{q}_\kappa^*>0$, $c_\kappa^*\in[0,\sqrt2)$  and a decreasing map $\mathfrak{c}:[0,\mathfrak{q}_\kappa^*)\to(c_\kappa^*,\sqrt{2}]$ satisfying $P(u_{\mathfrak{c}(\mathfrak{q}),\kappa})=\mathfrak{q}$. Then also, for every $\mathfrak{q}\in(-\mathfrak{q}_\kappa^*,\mathfrak{q}_\kappa^*)$ the infimum in \eqref{Emin} is attained, with $\mathcal{E}_\kappa(\mathfrak{q})=E_\kappa(u_{\mathfrak{c}(|\mathfrak{q}|),\kappa})$, the dark soliton $u_{\mathfrak{c}(|\mathfrak{q}|),\kappa}$ in the unique minimizer up to invariances and for every $|\mathfrak{q}|> \mathfrak{q}_\kappa^*$, we have $\mathcal{E}_\kappa(\mathfrak{q})=E_\kappa(u_{0,\kappa})$ and this infimum is not attained. Finally, for every $c\in(c_\kappa^*,\sqrt{2})$ the dark soliton $u_{c,\kappa}$ is orbitally stable for $d_\mathcal{X}$.
		\end{enumerate}
	\end{theorem}
	Theorem~\ref{thm:min} implies that dark solitons $u_{c,\kappa}$ with $0<c<c_\kappa^*$ are not global minimum of the energy at fixed momentum. Thus other methods must be employed to ensure stability. We display the energy-momentum of the dark soliton parametrized by $c$ in the left panel of Figure~\ref{fig1}--\ref{fig2}, it coincides with the function $\mathcal{E}_\kappa(\cdot)$ up to $\mathfrak{q}_\kappa^*$.
	By orbital stability for a distance $d$, we mean stability with respect to the orbit described by translations and rotations of the dark soliton $u_{c,\kappa}$.
	\begin{definition}\label{def:stabdef}\begin{enumerate}
			Let $s>2+1/2$ and $d$ be a distance on $\mathcal{X}(\R)$. The stability and instability for dark solitons is defined as follows.
			\item\label{item:stab} We say that $u_{c,\kappa}$ is orbitally stable for $d$ if for all $\varepsilon>0$, there exists $\delta>0$ with the following property:
			for every $\phi_0\in  H^{s}(\R)$ such that $d(u_{c,\kappa},u_{c,\kappa}+\phi_0)\leq\delta$, then, denoting $T=T_{\phi_0}$ the maximal time of existence of the solution $\Psi(x,t)=u_{c,\kappa}(x)+\phi(x,t)$ to \eqref{QGP} with $\phi\in\boC((-T,T); H^{s}(\R))$, we have
			\begin{equation}\label{def:stab}
				\sup_{t\in(-T,T)}\inf_{(x_0,\varphi)\in\R^2}d(u_{c,\kappa}(\cdot-x_0)e^{i\varphi},u_{c,\kappa}(\cdot)+\phi(\cdot,t))\leq\varepsilon.
			\end{equation}	
			\item\label{item:unstab}Assuming furthermore that the solution is global, we say that $u_{c,\kappa}$ is orbitally unstable for $d$ if there exists $\varepsilon_0>0$ such that for every $\delta>0$ we can find $\phi_0\in H^{s}(\R)$ so  that $d(u_{c,\kappa},u_{c,\kappa}+\phi_0)\leq\delta$ and the solution $\Psi(x,t)=u_{c,\kappa}(x)+\phi(x,t)$ to \eqref{QGP} with $\phi\in\boC(\R; H^{s}(\R))$ satisfies the following:
			there exists a constant $t=t(\delta,\epsilon_0)$ such that 
			\begin{equation}
				\inf_{(x_0,\varphi)\in\R^2}d(u_{c,\kappa}(\cdot-x_0)e^{i\varphi},u_{c,\kappa}(\cdot)+\phi(\cdot,t))>\varepsilon_0.
			\end{equation}
		\end{enumerate}
	\end{definition}
	To our knowledge, the results concerning the well-posedness of \eqref{QGP} are limited to local well-posedness (see Section~2). Therefore to make sense of orbital instability, we place ourselves in the case where blowup does not occur in the sense that the solution $\Psi$ to \eqref{QGP} starting at time $t=0$ close to the dark soliton $u_{c,\kappa}$ in Definition~\ref{def:stabdef}~\eqref{item:unstab} exists for all $t\in\R$ in the space $u_{c,\kappa}+ H^{s}$.

	The main result presented is the rigorous justification of the VK stability criterion for equation \eqref{QGP}.
	\begin{theorem}\label{thm:critstab}
		Let $\kappa<1/2$ and $0 <c<\sqrt{2}$. Let $u_{c,\kappa}\in\mathcal{X}(\R)\cap C^2(\R)$ be the unique solution to \eqref{TWc} up to invariances given by Theorem~\ref{thm:classiftw}. Then $$c\mapsto P(u_{c,\kappa})\in\boC^1((0,\sqrt{2})),$$ and the traveling-wave solution $u_{c,\kappa}$ is orbitally  stable for $d_\mathcal{X}$ if \begin{equation}\label{eq:critstab}
			\frac{dP(u_{c,\kappa})}{dc}<0,
		\end{equation} whereas it is unstable for $d_\mathcal{X}$ if blowup does not occur in $u_{c,\kappa}+ H^{s}(\R)$ for some $s>2+1/2$ and
		\begin{equation}\label{eq:critunstab}
			\frac{dP(u_{c,\kappa})}{dc}>0.
		\end{equation}
	\end{theorem}
	As explained in Section~\ref{sec:hamilton}, conditions \eqref{eq:critstab} and \eqref{eq:critunstab} amount to $d''(c)>0$ and $d''(c)<0$ respectively,
	where $d:c\mapsto E_\kappa(u_{c,\kappa})-cP(u_{c,\kappa})$ is the scalar function characterizing the stability in~\cite{grillakisshatah}.
	In Section~5 in \cite{delaire2023exotic}, we were able to compute explicitly the function $(c,\kappa)\mapsto P(u_{c,\kappa})$ and analyze its variations. Hence, we can restate the stability result in Theorem~\ref{thm:critstab} in terms of the propagation speed $c\in(0,\sqrt2)$ of the dark soliton follows.
	\begin{corollary}\label{cor:stab}
		There exists $\kappa_0<0$ with $\kappa_0\approx-3.636$ such that the two following statements hold
		\begin{enumerate}
			\item If $\kappa \in[\kappa_0,1/2)$, then $u_{c,\kappa}$ is orbitally stable for every $c\in(0,\sqrt{2})$.
			\item If $\kappa<\kappa_0$, then there exists $\tilde{c}_\kappa\in(0,\sqrt 2)$ such that $u_{c,\kappa}$ is orbitally unstable if $c\in(0,\tilde{c}_\kappa)$, whereas it is orbitally stable if $c\in(\tilde{c}_\kappa,\sqrt{2})$.
		\end{enumerate}
	\end{corollary}
	When $\kappa_0\leq\kappa<1/2$ we recover that the whole branch is stable as in the semilinear case $\kappa=0$ (see Theorem~4 in \cite{bethuel2008existence}). In the left and center panels of Figure~\ref{fig1} and Figure~\ref{fig2} respectively, we displayed the momentum of the dark soliton as a function of $c$ which yielded the stability conjecture proven in Corollary~\ref{cor:stab}. Figure~\ref{fig2} also depicts the value of $\title{c}_\kappa$ which satisfies $\partial_cP(u_{\tilde{c}_\kappa,\kappa})=0$ and that we approach using Newton's method.
	\begin{figure}[ht!]
		\centering
		\begin{tabular}{cc}
			\resizebox{0.35\textwidth}{!}{
				\begin{overpic}
					[scale=0.4,trim=0 0 0 0,clip]{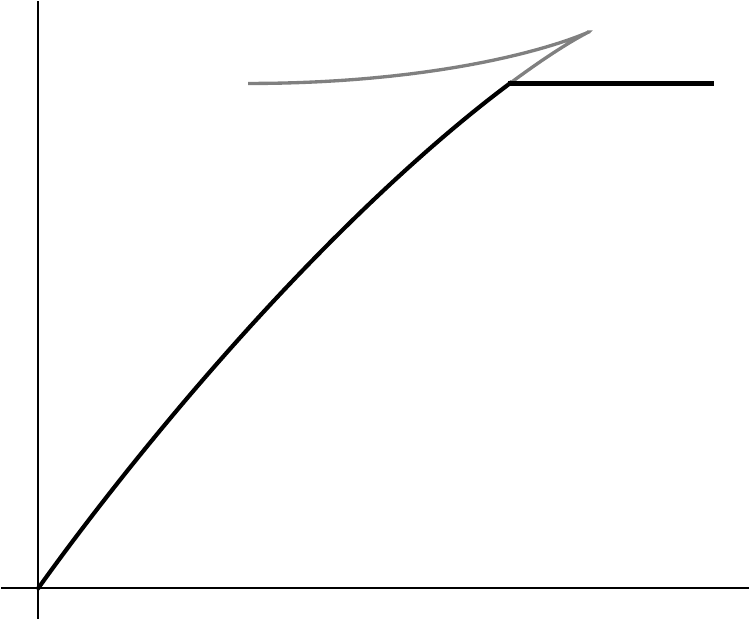}
					\put(25,68){\vector(3,1){7}}
					\put(6,65){{$c\to0$}}
					\put(20,10){\vector(-3,-1){10}}
					\put(22,10){{$c\to\sqrt{2}$}}
					\put(75,68){\vector(-3,1){7}}
					\put(78,66){{$c_\kappa^*$}}
					\dottedline{2}(68,70.5)(68,5.2)
					\put(65,-0.5){$\mathfrak{q}_\kappa^*$}
					\dottedline{2}(5,71.5)(70,71.5)
					\put(-12,70){3.748}
					\put(87,79){\vector(-1,0){7}}
					\put(90,78){$\tilde{c}_\kappa=0.473$}
					\put(92,8){$P(u_{c,\kappa})$}
					\put(6,80){$E_\kappa(u_{c,\kappa})$}
					\put(36,-1){$\frac{\pi}{2}$}
					\put(38.5,5){\line(0,1){2}}
				\end{overpic} 
			}
			&\qquad\qquad\resizebox{0.4\textwidth}{!}{
				\begin{overpic}
					[scale=0.4,trim=0 0 0 0,clip]{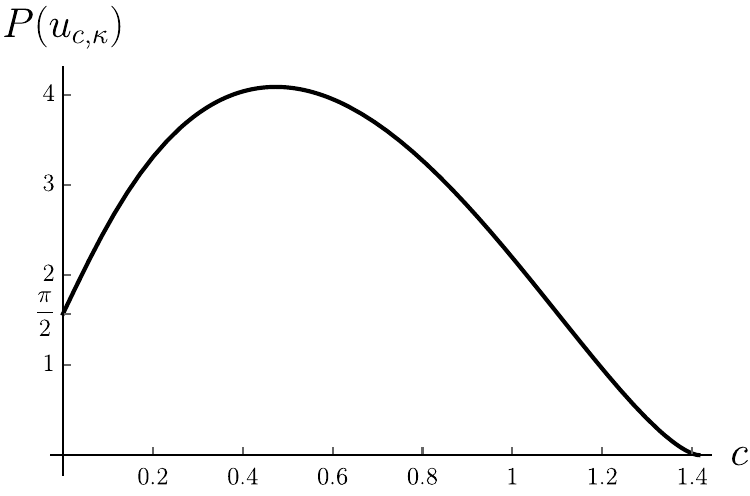}
				\end{overpic}
			}
		\end{tabular}
		\caption{Left panel depicts in grey the energy-momentum diagram of dark solitons $u_{c,\kappa}$ with $\kappa=-50$ it also displays in black the curve $\mathcal{E}_\kappa(\cdot)$. Right panel displays the momentum of the dark soliton $u_{c,\kappa}$ with $\kappa=-50$ as a function of the speed $c$.}
		\label{fig1}
	\end{figure}
	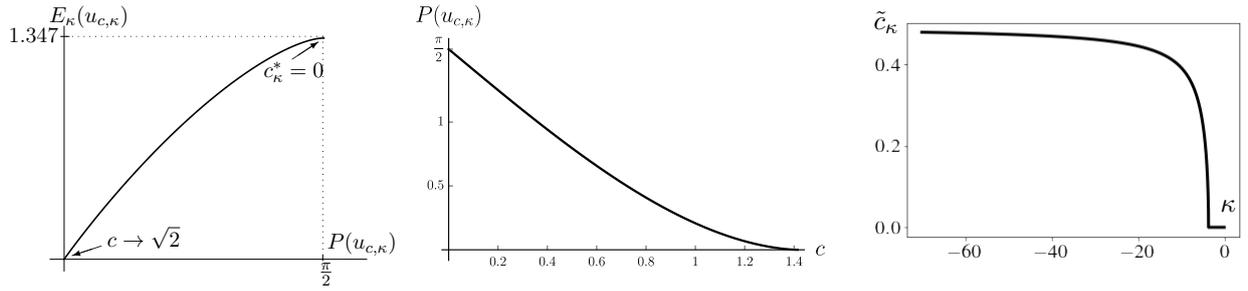
\begin{figure}[ht!]
		\centering
		\begin{tabular}{ccc}
			\resizebox{0.25\textwidth}{!}{
				\begin{overpic}
					[scale=0.4,trim=0 0 0 0,clip]{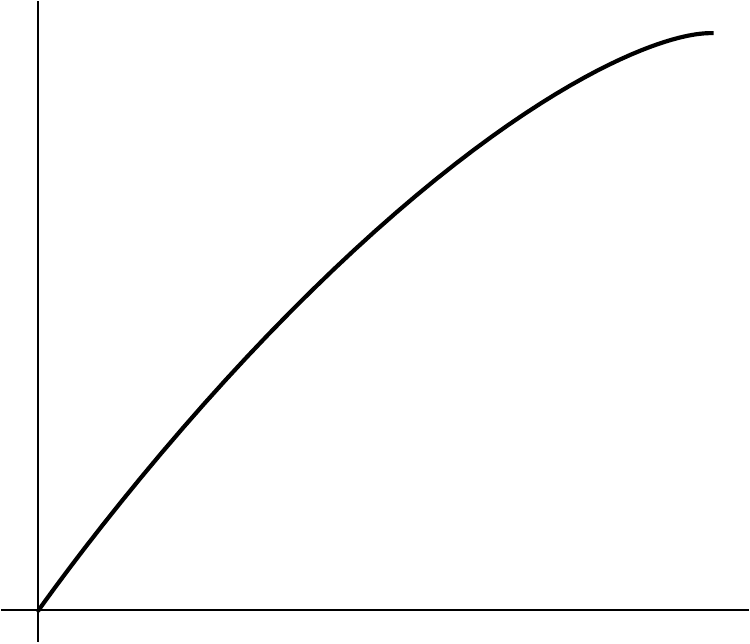}
					\put(-14,80){1.347}
					\put(4,82){\line(1,0){3}}
					\dottedline{2}(4,82)(94.5,82)
					\dottedline{2}(94.5,0)(94.5,82)      
					\put(84,74){\vector(3,2){9}}
					\put(74,68){{$c_\kappa^*=0$}}
					\put(18,9){\vector(-3,-1){10}}
					\put(20,9){{$c\to\sqrt{2}$}}
					\put(96,8){$P(u_{c,\kappa})$}
					\put(0,87){$E_\kappa(u_{c,\kappa})$}
					\put(93,-2){$\frac{\pi}{2}$}
					\put(94.5,4){\line(0,1){2}}
					\put(99.8,4.45){\line(1,0){10}}
				\end{overpic}
			}
			&\quad
			\resizebox{0.35\textwidth}{!}{
				\begin{overpic}
					[scale=0.6,trim=0 0 0 0,clip]{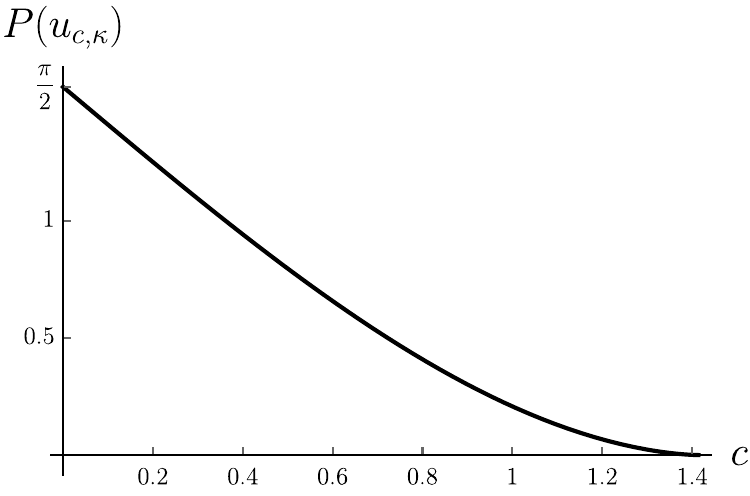}
				\end{overpic}
			}&\resizebox{0.35\textwidth}{!}{
				\begin{overpic}
					[scale=0.4,trim=20 0 0 0,clip]{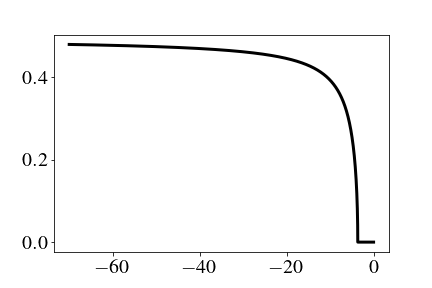}
					\put(0,60){$\tilde{c}_\kappa$}
					\put(85,15){$\kappa$}
				\end{overpic}
			}
		\end{tabular}
		\caption{Left panel depicts the energy-momentum diagram of the dark soliton $u_{c,\kappa}$ with $\kappa=-3$. The center panel represents the momentum of the dark soliton $u_{c,\kappa}$ with $\kappa=-3$ as a function of the speed $c$. The right panel displays the value of $\Tilde{c}_\kappa$ depending on $\kappa$, we put $\tilde{c}_\kappa=0$ for $\kappa\geq\kappa_0$.}
		\label{fig2}
	\end{figure}
	
	From now on, we fix $\kappa<1/2$ and remove this variable from the subscripts when no confusion can arise. Writing $\Psi=\sqrt{1-\eta}e^{i\theta}$ with $|\Psi|>0$ in $\R$, we can rewrite the energy and momentum in terms of $\eta$ and $v=\partial_x\theta$ and define other useful conservation laws for the hydrodynamic system \eqref{eq:hydrot}
	\begin{align}
		E_\kappa(\sqrt{1-\eta}e^{i\theta})=	E((\eta,v))&=\int_\R\frac{(\eta_x)^2}{4(1-\eta)}+v^2(1-\eta)+\frac{1}{2}(\eta^2-\kappa(\eta_x)^2),\quad\text{with }v=\partial_x\theta\quad(\mathbf{energy}),\label{Ehydro}\\
		P(\sqrt{1-\eta}e^{i\theta})=P((\eta,v))&=-\int_\R\eta v,\quad\text{with }v=\partial_x\theta\quad(\mathbf{momentum}),\label{phydro}\\	
		N((\eta,v))&=-\int_\R\eta,\quad(\mathbf{number~of~particles}),\label{pnumb}\\
		\Theta((\eta,v))&=\int_\R v,\quad(\mathbf{twisting~angle})\label{twist}.
	\end{align}
	We will check in Section~\ref{sec:hamilton} that in the space $X= H^1(\R)\times L^2(\R)$, we can recast \eqref{eq:hydrot} as the Hamiltonian system 
	\begin{equation}\label{eq:hamilton}
		\partial_t(\eta,v)=JE'(\eta,v),\text{ where }J=\begin{pmatrix}
			\partial_x&0\\
			0&\partial_x
		\end{pmatrix}.
	\end{equation}
	{Here $J:D(J)\to X$ is defined on a dense domain $D(J)\subset X^*= H^{{-1}}(\R)\times L^2(\R)$ and $J$ is not onto as explained in Remark~\ref{rem:onto}.}
	When $c>0,$ since the traveling wave $u_{c}(x-ct)=\sqrt{1-\eta_c(x-ct)}e^{i\theta_c(x-ct)}$ is smooth and rapidly decaying (in $(x,t)$), all the computations \eqref{eq:hydrot} and \eqref{Ehydro}--\eqref{twist} become rigorous and integrating in space \eqref{eq:hydrot}, we get for all $x\in\R$
	\begin{equation}\label{eq:twhydro}
		\left\{\begin{aligned}
			&-c\eta_c(x)=2v_c(x)(1-\eta_c(x)),\\
			&-cv_c(x)=-\left(\frac{\partial_{xx}\eta_c(x)}{2(1-\eta_c(x))}+\frac{(\partial_x\eta_c(x))^2}{4(1-\eta_c(x))^2}+v_c^2-\eta_c(x)-\kappa\partial_{xx}\eta_{c}(x)\right),
		\end{aligned}\right.
	\end{equation}
	where $v_c=\partial_x\theta_c$. System~\eqref{eq:twhydro} in its weak form can also be written as 
	\begin{equation}\label{eq:boundstatx}
		cP'((\eta_c,v_c))=E'((\eta_c,v_c)),\quad\text{ in } H^{{-1}}(\R)\times L^2(\R),
	\end{equation} 
	so that solving \eqref{eq:twhydro} amounts to find critical points of the Hamiltonian
	$E-cP.$
	
	The outline of this work follows that of \cite{linbubbles}.
	In Section~\ref{sec:hamilton} we check that \eqref{eq:hydrot} is a Hamiltonian system fitting the framework of~\cite{grillakisshatah} so that stability criterion \eqref{eq:critstab} holds for the hydrodynamical distance $d_{\mathrm{hy}}$ using their result. We deal with the instability result in Section~\ref{sec:instab}. Since $J$ is not onto, the instability result of~\cite{grillakisshatah} does not apply, we instead follow the construction of the energy-decreasing direction developped in~\cite{linbubbles}. In Section~\ref{sec:equiv}, we deduce the stability and instability for the energy distance $d_{\mathcal{X}}$ in Theorem~\ref{thm:critstab} and we show Corollary~\ref{cor:stab}. In Appendix~\ref{appendix}, we discuss some extensions of Theorem~\ref{thm:critstab}.
	\section{Hamiltonian system framework and orbital stability}\label{sec:hamilton}
	Let $X= H^1(\R;\R)\times  L^2(\R;\R)$ with the natural scalar product $\langle(\eta_1,v_1),(\eta_2,v_2)\rangle_X=\langle\eta_1,\eta_2\rangle_{ H^1(\R)}+\langle v_1,v_2\rangle_{ L^2(\R)}$. We also introduce the dual space $X^*= H^{{-1}}(\R;\R)\times  L^2(\R;\R)$ with duality product $$\langle\Lambda((\eta_1,v_1),(\eta_2,v_2)\rangle=\langle\eta_1,v_1),(\eta_2,v_2)\rangle_X,\quad\text{ for all }(\eta_1,v_1),(\eta_2,v_2)\in X,$$
	where $\Lambda=(1-\partial_{xx},1)$ is an isometric isomorphism between $X$ and $X^*$, but we do not identify $X$ to $X^*$. Regarding \eqref{Ehydro}, the space $X$ should be interpreted as the hydrodynamical counterpart of the energy space $\mathcal{X}(\R)$. 
	We have the variational triple $X\subset L^2(\R;\R)^2\subset X^*$ and we denote by $(\cdot,\cdot)_{ L^2(\R)}$ the natural scalar product on $ L^2(\R)$. For notational convenience, we shall use these products for $(\eta,v)\in X$ or solely in terms of $\eta\in H^1(\R)$ or $v\in L^2(\R)$, assuming the other coordinate to be zero. 
	In this setting, straightforward computations implies that $E$ and $P$ are $ C^2$-functional on $X$ with, in $X^*$,
	\begin{equation}\label{eq:Eprim}
		E'(\eta,v)=\left(-\frac{\partial_{x}\eta_x}{2(1-\eta)}-\frac{\eta_x^2}{4(1-\eta)^2}-v^2+\eta+\kappa\partial_{x}\eta_x,\quad2v(1-\eta)\right),
	\end{equation}
	so that \eqref{eq:hydrot} can be recast as the Hamiltonian system \eqref{eq:hamilton}.
	For every $(\delta\eta,\delta v)\in  H^1(\R)\times L^2(\R)$, we get in $X^*$
	\begin{align}\label{eq:Ehess}
		E''(\eta,v)(\delta\eta,\delta v)=2\Big(&	-\partial_x\left(\frac{1-2\kappa+2\kappa\eta}{4(1-\eta)}(\delta\eta)_x\right)-
		\partial_x\left(\frac{\eta_x}{4(1-\eta)^2}\right)\delta\eta+\frac{\eta_x^2}{4(1-\eta)^3}\delta\eta-v\delta v,\\
			&(1-\eta)\delta v-v\delta\eta\Big),
	\end{align}
	\begin{equation}\label{eq:pprim}
		P'(\eta,v)=-(v,\eta),\quad\text{ and }\quad P''(\eta,v)(\delta\eta,\delta v)=-(\delta v,\delta\eta).
	\end{equation}
	Let $\tau$ denote the unitary group of translation on $X$ with, for $x_0\in\R$
	\begin{equation}\label{def:tau}
		\tau(x_0)(\eta,v)=(\eta,v)(x-x_0),
	\end{equation}
	\begin{flalign}
		\text{\raggedleft so that,}\hspace{2cm}\tau'(0)=\begin{pmatrix}
			-\partial_x&0\\
			0&-\partial_x
		\end{pmatrix},\quad\text{ and }\quad\tau'(x_0)=\tau(x_0)\tau'(0)=\tau'(0)\tau(x_0),&&
	\end{flalign} 
	are closed operators in $X$ with domain $D(\tau'(\cdot))=D(\tau'(0))=H^2(\R)\times H^1(\R)$.
	Also, letting \begin{equation}\label{eq:B}
		B=\begin{pmatrix}
			0&-1\\
			-1&0
		\end{pmatrix}\in\mathcal{L}(X;X^*),
	\end{equation}
	we get $P((\eta,v))=\langle B(\eta,v),(\eta, v)\rangle/2$ and $JB$ is an extension of $\tau'(0)$. We call bound state a solution to \eqref{eq:hydrot} of the form 
	\begin{equation}\label{def:boundstatt}
		(\eta,v)(x,t)=\tau(ct)(\eta_c(x),v_c(x)),\quad\text{ for all }(x,t)\in\R\times\R,
	\end{equation}
	such that $(\eta_c,v_c)\in X$ are critical points of $E-cP$ i.e.\
	\begin{equation*}
		cP'((\eta_c,v_c))=E'((\eta_c,v_c)).	
	\end{equation*} 
	It will also be of importance to perform spectral analysis on the Hessian 
	\begin{equation}\label{def:hess}
		H_c=E''((\eta_c,v_c))-cP''((\eta_c,v_c)),
	\end{equation} 
	in this regard, a demonstration of every abstract result we use can be found either in this textbook by C.~Cheverry and N.~Raymond~\cite{spectraltheorybook} or in Appendix~B of the book by J.Angulo Pava~\cite{angulospectral}.
	Let us recall the three assumptions in \cite{grillakisshatah} that ensure their stability theorem.

	\textit{Assumption~1}{~(Existence of solutions).} For each $(\eta_0,v_0)\in  H^{s}(\R)\times H^{{s-1}}(\R)$ with $s>2+1/2$, there exist $T>0$ and a solution $(\eta,v)\in\boC((-T,T); H^{s}(\R)\times H^{{s-1}}(\R))$ to \eqref{eq:hydrot}  such that $(\eta,v)(\cdot,0)=(\eta_0,v_0)$ and the energy and momentum are conserved in time
	\begin{equation}\label{eq:cons}
		E_\kappa((\eta,v)(\cdot,t))=E_\kappa((\eta_0,v_0)),\quad\text{ and} \quad P((\eta,v)(\cdot,t))=P((\eta_0,v_0)).
	\end{equation} 
	If $(\eta_c,v_c)$ is a bound state in $H^s(\R)\times H^{s-1}(\R)$, it suffices to have the above for $(\eta_0,v_0)$ in a neighborhood of $(\eta_c,v_c)$ in $ ( H^{s}(\R)\times H^{{s-1}}(\R),||\cdot||_X).$
	
	\textit{Assumption~2}{~(Existence of bound states).} There exist constants $c_1<c_2$ and a $\boC^1((c_1,c_2);X)$ map $c\mapsto(\eta_c,v_c)$ with the following properties. 
	\begin{enumerate}
		\item $E'(\eta_c,v_c)=cP'(\eta_c,v_c)$.
		\item $(\eta_c,v_c)\in D(\tau'(0)^3)\cap D(J\Lambda \tau'(0)^2)$, where $\Lambda=(1-\partial_{xx},1)$ is the isomorphism between $X$ and $X^*$.
		\item $\tau'(0)(\eta_c,v_c)\ne0$.
	\end{enumerate}

	\textit{Assumption~3}{~(Spectral decomposition of $H_c$).} For every $c\in(c_1,c_2)$ and $H_c$ given by \eqref{def:hess}, we assume
	\begin{enumerate}
		\item There exists $\chi\in X$ such that 
		\begin{equation}
			\langle H_c\chi,\chi\rangle<0.
		\end{equation}
		\item There exists a closed subspace $P_c\subset X$, and $C>0$ such that
		\begin{equation}
			\langle H_cp,p\rangle\geq C||p||^2,\quad\text{ for all }p\in P_c.
		\end{equation}
		\item For every $u\in X$, there exist unique $(a,b)\in\R^2$ and $p\in P_c$ such that
		\begin{equation}\label{eq:decompo}
			u=a\chi+b\tau'(0)(\eta_c,v_c)+p.
		\end{equation}
	\end{enumerate}
	\begin{remark}
		To our knowledge, when $\kappa\ne0$, there are no result for the IVP associated with \eqref{eq:hydrot} with initial data in $X$, thus
		\textit{Assumption~1} agrees with our definition of orbital stability with respect to $ H^{s}(\R)$-perturbations~\eqref{def:stab}.
		Also, it is worth mentioning that \textit{Assumption~3} corresponds \textit{Assumption~3B} in \cite{grillakisshatah} i.e. the decomposition \eqref{eq:decompo} is not necessarily the orthogonal decomposition; this enables them to handle Hamiltonian systems in Banach spaces.
	\end{remark}
	Recall that \eqref{eq:hydrot} can be rewritten as an Euler--Korteweg system for which the local well-posedness was investigated in \cite{benzoniLWP}. Applying the result therein to \eqref{QGP} for $\kappa<1/2$ (see \cite{delaire2023exotic} for details), we deduce the following, from which \textit{Assumption~1} follows in a neighborhood of $(\eta_c,v_c)$ by Lemma~\ref{lem:dequiv}.
	\begin{theorem}[Theorem~5.1 in \cite{benzoniLWP}]\label{thm:lwpEK}
		Let $s>1/2+2$, $\kappa<1/2$ and $c>0$ and \begin{equation*}
			J_\kappa=\begin{cases}
				(0,\infty)\text{ if }\kappa\leq0,\\
				(0,1/\sqrt{2\kappa}),\text{ if }0<\kappa<1/2.
			\end{cases}
		\end{equation*}
		If $\Psi_0\in u_{c,\kappa}+ H^{s}(\R)$ satisfies $|\Psi_0(x)|\in J_\kappa$ for every $x\in\R$,
		then there exists $T_{\Psi_0}>0$, the maximal time of existence such that for every $T\in (0,T_{\Psi_0})$, \eqref{QGP} has a unique solution $\Psi$ on  $\R\times[0,T]$, satisfying $\Psi(\cdot,0)=\Psi_0$ and 
		\begin{equation*}\label{lwp}
			\left\{ \begin{aligned}
				&\Psi\in\boC([0,T];u_{c,\kappa}+ H^{s}(\R)),\\
				&|\Psi|(\R\times[0,T])\subset\subset J_\kappa.
			\end{aligned}\right.
		\end{equation*}
		Moreover, the flow map is continuous in a neighborhood of $\Psi_0$ in $u_{c,\kappa}+ H^{s}(\R)$, and the energy and momentum are conserved, i.e.\ \eqref{eq:cons} holds for all $t\in[0,T].$
		Also writing $\Psi=\sqrt{1-\eta}e^{i\theta}$ we deduce that $\eta$ and $v=\partial_{x}\theta$ satisfy \eqref{eq:hydrot} with $(\eta,v)\in\boC([0,T]; H^{s}(\R)\times H^{{s-1}}(\R))$.
	\end{theorem}
	For \textit{Assumption~2}, take $c\in(0,\sqrt2)$, and let $u_{c}$ be the dark soliton given by Theorem~\ref{thm:classiftw} so that $\inf_{x\in\R}|u_c|>0$. Then we can write $u_c=\sqrt{1-\eta_c}e^{i\theta_c}$ with $(\eta_c,\partial_x\theta_c)$ satisfying \eqref{eq:twhydro} pointwisely and \textit{Assumption~2}~(1) follows. By the smoothness of $u_c(x)$ with respect to $c$ and the decay estimate \eqref{eq:decaydark} and using the identity 
	\begin{equation}\label{eq:polarder}
		u_c'=\left(\frac{\eta_c'}{2\sqrt{1-\eta}}+i\sqrt{1-\eta}\theta_c'\right)e^{i\theta},
	\end{equation} it is easy to check that $c\mapsto(\eta_c,\partial_x\theta_c)$ belongs in $\boC^\infty((0,\sqrt2)),X)$. Once again appealing to the smoothness and decay of $u_c$, we recover $(\eta_c,\partial_x\theta_c)\in  H^\infty(\R)^2$ which implies \textit{Assumption~2}~(2). 
	Finally, since $u_c$ is nonconstant, \textit{Assumption~2}~(3) follows.

	To check \textit{Assumption~3}, it remains to compute $\langle H_c(\delta\eta,\delta v),(\delta\eta,\delta v)\rangle$. Using \eqref{eq:Ehess}--\eqref{eq:pprim}, we get
	\begin{align}\label{eq:Hfull}
		\langle H_c(\delta\eta,\delta v),(\delta\eta,\delta\notag v)\rangle=&\langle-\partial_x\left(\frac{1-2\kappa+2\kappa\eta}{2(1-\eta)}\partial_{x}\right)\delta\eta,\delta\eta\rangle
		-\langle\partial_{x}\left(\frac{\eta_x}{2(1-\eta)^2}\right)\delta\eta,\delta\eta\rangle\\
		&+\int_\R \frac{\eta_x^2}{2(1-\eta)^3}\delta\eta^2+
		\delta\eta^2
		+\int_\R-4v\delta v\delta\eta+2(1-\eta)\delta v^2+2c\delta v\delta\eta,\notag\\
		=&\langle-\partial_x\left(\frac{1-2\kappa+2\kappa\eta}{2(1-\eta)}\partial_{x}\right)\delta\eta,\delta\eta\rangle
		-\langle\partial_{x}\left(\frac{\eta_x}{2(1-\eta)^2}\right)\delta\eta,\delta\eta\rangle
		+\int_\R \frac{\eta_x^2}{2(1-\eta)^3}\delta\eta^2\notag\\&+
		\delta\eta^2
		-\int_\R\frac{(c-2v)^2}{2(1-\eta)}\delta\eta^2+\int_\R2(1-\eta_c)\left(\delta v+\frac{c-2v}{2(1-\eta)}\delta\eta\right)^2.
	\end{align} 
	Using the first line in \eqref{eq:twhydro} and taking $(\eta,v)=(\eta_c,v_c)$ in \eqref{eq:Hfull}, we can rewrite the last line in \eqref{eq:Hfull} to obtain
	\begin{equation}\label{eq:H}
		\langle H_c(\delta\eta,\delta v),(\delta\eta,\delta v)\rangle=\langle L_c\delta\eta,\delta\eta\rangle+\int_\R2(1-\eta_c)\left(\delta v+\frac{c}{2(1-\eta_c)^2}\delta\eta\right)^2,
	\end{equation}
	where $L_c$ is an self-adjoint unbounded operator on $ L^2(\R;\R)$ with domain $ H^2(\R;\R)$ given by
	\begin{equation}\label{eq:L}
		L_c=-\partial_x\left(\frac{1-2\kappa+2\kappa\eta_c}{2(1-\eta_c)}\partial_{x}\right)-\frac{\partial_{xx}\eta_c}{2(1-\eta_c)^2}-\frac{(\partial_x\eta_c)^2}{2(1-\eta)^3}+1-\frac{c^2}{2(1-\eta_c)^3}.
	\end{equation}
	Since $L_c$ is semi-bounded from below, integrating by part the first term in $\langle L_c\delta\eta,\delta\eta\rangle$, we infer that $L_c$ is associated with a quadratic form on $ H^1(\R;\R)$ by the Lax--Milgram theorem (Theorem~2.89 in \cite{spectraltheorybook}).
	
	Since $0<\eta_c\leq1-c^2/2$ and $0<c<\sqrt{2}$, we have  $1-\eta_c>0$ and $1-2\kappa+2\kappa\eta_c>0$ in $\R$. Then, using ideas similar to the ones in section~6.3.2  of \cite{spectraltheorybook} and Lemma~6.13 in \cite{spectraltheorybook}, we have in $ L^2(\R;\R)$
	\begin{equation}\label{eq:spectrlaplac}
		\sigma_{\mathrm{ess}}\left(-\partial_x\left(\frac{1-2\kappa+2\kappa\eta_c}{2(1-\eta_c)}\partial_x\right)\right)=[0,\infty).
	\end{equation}
	We infer that the other terms in $L_c-\lambda$ are compact operators for some $\lambda>0$. Indeed, notice that 
	\begin{equation}\label{lim:compact}
		\lim_{|x|\to \infty}-\frac{\partial_{xx}\eta_c(x)}{2(1-\eta_c(x))^2}-\frac{(\partial_x\eta_c(x))^2}{2(1-\eta_c(x))^3}+1-\frac{c^2}{2(1-\eta_c(x))^3}=1-\frac{c^2}{2}>0.
	\end{equation}
	\begin{lemma}\label{lem:Tcompact}
		Let $0<c<\sqrt{2}$ and fix $\lambda=1-c^2/2>0$. Let also $T_c\in\mathcal{L}( H^2(\R;\R), L^2(\R;\R))$ be given by
		\begin{equation}
			T_c=-\frac{\partial_{xx}\eta_c}{2(1-\eta_c)^2}-\frac{(\partial_x\eta_c)^2}{2(1-\eta_c)^3}+1-\frac{c^2}{2(1-\eta_c)^3}
		\end{equation}
		then $T_c-\lambda$ is compact i.e. for every sequence  $(u_n)\subset H^2(\R)$ with $u_n\rightharpoonup0$ we can extract a subsequence such that $(T_c-\lambda)(u_{n_k})\to0$ in $ L^2(\R)$. 
	\end{lemma}
	\begin{proof}
		Let $\varepsilon>0$ and $R=R_\varepsilon>0$.  by   Kato--Rellich theorem (Theorem~4.20 in \cite{spectraltheorybook}) there exists $v\in H^1((-R,R))$ and a subsequence $u_{n_k}\to v$ in $H^1((-R,R))$. Of course $v=0$ by the uniqueness of the weak limit.
		Then denoting $I_R=\R\backslash(-R,R)$  we split  $||(T_c-\lambda)(u_{n_k})||_{ L^2(\R)}^2$ as 
		\begin{equation}\label{eq:LtwoT}
			||(T_c-\lambda)(u_{n_k})||_{ L^2(\R)}^2=||(T_c-\lambda)(u_{n_k})||_{ L^2(I_R)}^2+||(T_c-\lambda)(u_{n_k})||_{ L^2((-R,R))}^2.
		\end{equation}
		We deduce, using \eqref{lim:compact} that taking $R>0$ big enough yield $$||T_c-\lambda||^2_{\mathcal{L}( H^2(I_R); L^2(I_R))}\leq\varepsilon/(2\sup_{k\in\N}||u_{n_k}||^2_{ H^2(\R)}),$$
		letting us bound the first term in \eqref{eq:LtwoT} by $\varepsilon/2$.
		Once $R$ is fixed, we take $k$ big enough so that considering $T_c-\lambda$ as an operator in $\mathcal{L}( H^2((-R,R)); L^2((-R,R)))$, we have 
		$||(T_c-\lambda)(u_{n_k})||^2_{L^2(-R,R)}\leq\varepsilon/2$. This shows that $(T_c-\lambda)(u_{n_k})\to0$ in $ L^2(\R)$ ending the proof.
	\end{proof}
	Lemma~\ref{lem:Tcompact}, enables us to compute $\sigma_{\mathrm{ess}}(L_c)=\sigma_{\mathrm{ess}}(\partial_x((1-2\kappa+2\kappa\eta_c)/(2-2\eta_c)\partial_x)+\lambda+T_c-\lambda)$ which yields, using a generalized version of the Weyl essential spectrum theorem (see Theorem~B.48 in \cite{angulospectral})
	\begin{equation}\label{eq:spectrL}
		\sigma_{\mathrm{ess}}(L_c)=[\lambda,\infty)=[1-c^2/2,\infty).
	\end{equation}
	The operator $L_c$ being self-adjoint, the rest of its spectrum $\sigma_{\mathrm{fred}}(L_c)=\sigma(L_c)\backslash\sigma_{\mathrm{ess}}(L_c)$ lies in $(-\infty,1-c^2/2)$ and every $\mu\in\sigma_{\mathrm{fred}}(L_c)$ is an isolated eigenvalue of finite multiplicity (see Section~6.3 in \cite{spectraltheorybook}). 
	Using the first line in \eqref{eq:twhydro} we can rewrite second equation in \eqref{eq:twhydro} only in terms of $\eta_c$
	\begin{equation}\label{eq:etac}
		-\left(\frac{1-2\kappa+2\kappa\eta_c}{2(1-\eta_c)}\partial_{x}\partial_x\eta_c\right)-\frac{(\partial_x\eta_c)^2}{4(1-\eta)^2}+\eta_c-c^2\frac{\eta^2-2\eta}{4(1-\eta)^2},\quad\text{ in }\R.
	\end{equation}
	Differentiating \eqref{eq:etac} with respect to $x$ yields
	$\partial_x\eta_c\in \ker(L_c)$. Since $|u_c|'>0$ in $(0,\infty)$ and $|u_c|$ is even, we deduce that $\partial_x\eta_c=\partial_x(1-|u_c|^2)$ vanishes once at $x=0$, so that, by the Sturm-Liouville oscillation theorem (see Appendix~B.5 in \cite{angulospectral}), there exists one and only one negative eigenvalue $\mu_-<0$ of multiplicity one with eigenvector $\chi_1\in H^2(\R)$.
	The eigenvalues $\mu_-$ and $0$ being simple and isolated, by the min-max principle (see Section~6.4 in \cite{spectraltheorybook}), we infer that there exists $\mu_1>0$ such that for every $z\in H^1(\R)$ satisfying
	\begin{equation}\label{eq:ortheta}
		(z,\chi_1)_{ L^2(\R)}=(z,\partial_x\eta_c)_{ L^2(\R)}=0,
	\end{equation} we have
	$\langle L_cz,z\rangle\geq\mu_1||z||_{ L^2(\R)}^2. $
	It is also clear that $L_c$ satisfies some kind of G$\mathring{\mathrm{a}}$rding inequality so that proceeding along the same lines as in Lemma~4.1 in \cite{linbubbles} yields the existence of $\mu>0$ such that $\langle L_cz,z\rangle\geq\mu||z||_{ H^1(\R)}^2.$
	Setting \begin{equation}\label{def:chi-}
		\chi_-=\left(\chi_1,-\frac{c}{2(1-\eta)^2}\chi_1\right),
	\end{equation}
	and using \eqref{eq:H}, we get $\langle H_c\chi_-,\chi_-\rangle=\langle L_c\chi_1,\chi_1\rangle<0$. Then, letting $$P_c=\{(p_1,p_2)\in X:(p_1,\chi_1)_{ L^2(\R)}=(p_1,\partial_x\eta_c)_{ L^2(\R)}=0\},$$
	and considering the $ L^2(\R)$-orthogonal decomposition of $\eta=a\chi_1+b\partial_x\eta+p_1$,
	we can perform the following decomposition of $(\eta,v)\in X$: $(\eta,v)=a\chi_-+b\partial_x(\eta_c,v_c)+(p_1,p_2),$  
	with $p_2=v-ac\chi_1/(2-2\eta)-b\partial_xv_c$. The uniqueness of the decomposition of $(\eta,v)$ follows from the uniqueness of the orthogonal decomposition of $\eta$. Using Lemma~4.2 in \cite{linbubbles} (see also \cite{guosemilin}) we deduce that there exists $C>0$ satisfying $\langle H_cp,p\rangle>C||p||^2_X$ for all $p\in P_c$ so that \textit{Assumption~3} follows.

	Let $d(c)=E((\eta_c,v_c))-cP((\eta_c,v_c))$. Since $c\mapsto(\eta_c,v_c)$ is in $\boC^1((0,\sqrt{2}))$, we can compute
	\begin{equation}\label{eq:critgss}
		d''(c)=\left(\langle E'((\eta_c,v_c))-cP'((\eta_c,v_c)),\partial_c(\eta_c,v_c)\rangle-P((\eta_c,v_c))\right)'=-\frac{d}{dc}P((\eta_c,v_c)).
	\end{equation}
	Assuming \eqref{eq:critstab}, we get $d''(c)>0$ and we can apply Theorem~3 in \cite{grillakisshatah}. The only remaining thing to check is that for any sequence $(\eta^n,v^n)=(\eta,v)_n\subset X$ with $E((\eta,v)_n)\to E((\eta_c,v_c))$ and $P((\eta,v)_n)\to P((\eta_c,v_c)),$ we can construct $(u_n)\subset X$ such that, up to extraction, we have $||u_n-(\eta,v)_n||_X\to0$ and $P(u_n)=P((\eta_c,v_c))$. By the first line in \eqref{eq:twhydro}, we have $P((\eta_c,v_c))>0$ for all $c\in(0,\sqrt{2})$, thus, up to a subsequence extraction, we assume   $P((\eta,v)_n)>0$ for all $n\in\N$. Hence, $\alpha_n=P((\eta_c,v_c))/P((\eta,v)_n)$ is well defined with $\alpha_n\to1$ so that $u_n=(\eta_n,\alpha_nv_n)$ satisfies $P(u_n)=P((\eta_c,v_c))$ and we can check $||u_n-(\eta_n,v_n)||_X\to0$. For $\psi=\sqrt{1-\eta}e^{i\theta}$ with $\inf_\R|\psi|>0$, notice that we can always choose $\varphi\in\R$ so that $\mathrm{Arg}(\psi(0)/u_c(0))-\varphi=0$ thus 
	$$\inf_{\varphi\in\R}d_{\mathrm{hy}}(\psi,u_ce^{i\varphi})=||(\eta,v)-(\eta_c,v_c)||_X,$$ 
	so that the stability of $u_c$ for $d_{\mathrm{hy}}$ follows from the orbital stability of the bound state $(\eta_c,v_c)$ in $X$.
	\begin{theorem}[Theorem~3 in \cite{grillakisshatah}]\label{thm:stabgss}
		If Assumptions~1--3 hold, then the following statement is satisfied.
		If $d''(c)>0$ then there exists $\varepsilon>0$ such that $E(\cdot)$ is minimized at $\mathfrak u=(\eta_c,v_c)$ in the space 
		$$\{u\in X : \inf_{x_0\in\R}||T(x_0)u-\mathfrak u||_{X}<\varepsilon\text{ and }P(u)=P(\mathfrak u)\}.$$
		Moreover the dark soliton $u_c$ is orbitally stable for $d_{\mathrm{hy}}$.
	\end{theorem}
	\begin{remark}\label{rem:onto}
		Theorem~3 in \cite{grillakisshatah} is stated as an equivalence between the variation of $d(\cdot)$ and the stability. However, we cannot apply the converse statement since the skew-symmetric operator  $J$ in \eqref{eq:hamilton} is not onto. For instance $(\eta_c,v_c)\notin \mathrm{Range}(J)$ because  $\eta_c>0$ in $\R$ and hence $C+\int^x_{-\infty}\eta_c\notin  H^{{-1}}(\R)$ for every $C\in\R$.
	\end{remark}
	\section{Instability}\label{sec:instab}
	Regarding the instability, arguments similar to the ones used by Lin in~\cite{linbubbles} yield the result. We explain how to modify the first steps that consist of creating an energy-decreasing direction in $X$ while referring to \cite{linbubbles} for detailed demonstrations. From now on, we assume that $c\in(0,\sqrt{2})$ is such that the dark soliton $(\eta_c,v_c)\in X$ satisfies the VK inequality \eqref{eq:critunstab}. We also fix $s>2+1/2$ to be the regularity of the perturbations in Definition~\ref{def:stab}~\eqref{item:unstab}.   
	Firstly, we have $\chi_1\in  H^\infty(\R)$ by a bootstrap argument since it satisfies $\chi_1\in H^2(\R)$ and $L_c(\chi_1)=\mu_-\chi_1$, with $L_c$  given by \eqref{eq:L}. In fact, as in \cite{linbubbles}, we can assume that $\chi_-\in X_s$ with
	\begin{equation}\label{def:Xs}
		X_s=( H^{s}(\R)\times H^{{s}}(\R))\cap( L^1(\R,(1+|x|)dx)\times L^1(\R,(1+|x|)dx)),
	\end{equation} where $\chi_-$ is the negative direction of $H_c$ given by \eqref{def:chi-}. We also get $(\eta_c,v_c)\in X_s$ using the smoothness and rapid decay of the dark solitons established in Theorem~\ref{thm:classiftw}.  Using the implicit function theorem on $(q,l)\mapsto(\eta_q,v_q)+l\chi_-$, we can find a $\boC^1$-map, $l(q)$ and a curve $\Psi(q)$ defined in the neighborhood of $c$ with \begin{equation}\label{eq:minunstab}
		\Psi(q)=(\eta_q,v_q)+l(q)\chi_-,
	\end{equation}satisfying $P(\Psi(q))\equiv P((\eta_c,v_c))$. This allows us to build a negative direction for the Hessian $H_c$ given by \eqref{def:hess}  with some orthogonality property.
	\begin{lemma}[Corollary~5.1 in \cite{linbubbles}]\label{lem:y0}
		There exists $y_0\in X_s$ such that $\langle By_0,(\eta_c,v_c)\rangle=0$ and $\langle H_cy_0,y_0\rangle<0$, where $B$ is given by \eqref{eq:B}.
	\end{lemma}
	Now, we want to build a negative direction as in Lemma~\ref{lem:y0} for which the number of particles \eqref{pnumb} and twisting angle \eqref{twist} are zero. We use the following.
	\begin{lemma}\label{lem:orthint}
		For all $\eta\in H^2(\R)$ and $\nu\in\R$, there exists a sequence $(u_n)\subset X_s$ such that $\int_\R u_n=\nu$,~$u_n\to0$ in $ H^1(\R)$ and $\langle u_n,\eta\rangle_X=0$.
	\end{lemma}
	\begin{proof}
		The proof proceeds as in \cite{linbubbles}; let $\varphi\in\boC^\infty_c(\R) $ with $\int_\R\varphi=\nu$. First if $\eta\equiv0$, then taking $u_n(\cdot)=\varphi(\cdot/n)/n$ for any $\varphi$ defined above suffices.
		Assuming that $\eta\ne0$, then by a duality and density argument there exists $\psi\in\boC^\infty_c(\R)$ such that $\langle \eta_x,\psi\rangle_{ H^1(\R)}\ne0$. Indeed, otherwise we would have $\eta_x=0$ but since $\eta\in L^2(\R)$ this contradicts $\eta\ne0$. 
		Letting $u_n(\cdot)=\varphi(\cdot/n)/n+\alpha_n\psi_x$, we readily get $u_n\in X_s$ and $\int_\R u_n=\nu$.
		Then, setting $\alpha_n=\langle\varphi(\cdot/n)/n,\eta\rangle_{ H^1(\R)}/\langle\eta_x,\psi\rangle_{ H^1(\R)},$ we obtain $\langle\eta,u_n\rangle_{ H^1(\R)}=0$. Finally, we deduce after rescaling that $\alpha_n\to0$ using $|\alpha_n||\langle\eta_x,\psi\rangle_{ H^1(\R)}|\leq||\eta||_{ H^1(\R)}||\varphi(\cdot/n)/n||_{ H^1(\R)}$ and also $u_n\to0$ in $ H^1(\R)$.
	\end{proof}
	Since $s>2+1/2$, we can apply Lemma~\ref{lem:orthint} to each component of $y_0=(\eta_0,v_0)$ where $y_0\in X_s$ is given by Lemma~\ref{lem:y0}.
	\begin{lemma}[Lemma~5.3 in \cite{linbubbles}]\label{lem:dirneg}
		For all $s>2+1/2$, there exists $y=(\eta_1,v_1)\in X_s$ such that $\langle By,(\eta_c,v_c)\rangle=0,$ $\langle H_cy,y\rangle<0$ and 
		$$\int_\R \eta_1=\int_\R v_1=0.$$
	\end{lemma}
	Following along Section~5 in \cite{linbubbles} we can deduce that $(\eta_c,v_c)$ is not a minimizer of the energy at fixed momentum and that it is unstable. More precisely, let $(\eta,v)=\Psi(q)$ where $\Psi(\cdot)$ is given by \eqref{eq:minunstab} and $q$ is in the vicinity of $c$. Taking $\psi=\sqrt{1-\eta}e^{i\theta}$ with $\theta(x)=\theta_c(0)+\int_0^xv$, we can prove that $\psi$ is arbitrarily close to $u_{c,\kappa}$ for $d_{\mathrm{hy}}$ as $q\to c$ with $q\ne c$. This implies that eventually as $q\to c$ we have $|\psi|(\R)\subset\subset J_\kappa$ and $\psi\in u_{c}+ H^{{s}}$ so that there exists a solution to \eqref{QGP} starting at $\psi$ at $t=0$. The instability of $u_c$ for $d_{\mathrm{hy}}$ follows as in \cite{linbubbles}. 
	\begin{theorem}[Theorem~5.1 in \cite{linbubbles}]\label{thm:unstablin}
		Let $c\in(0,\sqrt{2})$, $s>2+1/2$ and assume that $(\eta_c,v_c)\in X_s$ is the dark soliton. If $\partial_cP((\eta_c,v_c))>0,$ then the curve $\Psi(\cdot)$ given by \eqref{eq:minunstab} defined in the vicinity of $c$ satisfies 
		$P(\Psi(\cdot))\equiv P((\eta_c,v_c))$ and $E(\Psi(q))<E((\eta_c,v_c)),$ when $q\ne c$. 
		Moreover the dark soliton $u_c$ is unstable for $d_\mathrm{hy}$.
	\end{theorem}
	
	\section{Stability in the energy distance $d_\mathcal{X}$}\label{sec:equiv}
	Let $\kappa<1/2$ and $c\in(0,\sqrt2)$, we first show that the stability and instability results stated in terms of $d_{\mathrm{hy}}$ given by \eqref{def:metric2} in Theorem~\ref{thm:stabgss} and in Theorem~\ref{thm:unstablin} respectively, are equivalent to the stability and instability result in Theorem~\ref{thm:critstab} in terms of the distance $d_\mathcal{X}$ given by \eqref{def:metric} respectively; thus proving Theorem~\ref{thm:critstab}. For deeper insights on this equivalence we refer to Lemma~9--10 in \cite{chiron-stability}.
	\begin{lemma}\label{lem:dequiv}
		Let $\kappa<1/2$, $c\in(0,\sqrt{2})$ and $u_{c,\kappa}$ be a dark soliton given by Theorem~\ref{thm:classiftw}. Then there exists $K$ and $\varepsilon>0$ depending only on $u_{c,\kappa}$ such that for every $\Psi\in\mathcal{X}(\R)$, if $\Psi$ satisfies $d_\mathcal{X}(\Psi,u_{c,\kappa})<\varepsilon$, then we can write $\Psi=\sqrt{1-\eta}e^{i\theta}$ with $(\eta,\partial_x\theta)\in X$ and we have
		\begin{equation}
			\frac{1}{K}d_{\mathrm{hy}}(\Psi,u_{c,\kappa})\leq d_{\mathcal{X}}(\Psi,u_{c,\kappa})\leq Kd_{\mathrm{hy}}(\Psi,u_{c,\kappa}).
		\end{equation} 
	\end{lemma}
	\begin{proof}
		For clarity's sake, we recall the main arguments of the proof in \cite{chironstab}.
		We remove $\kappa$ in the subscript of $u_{c,\kappa}=u_c$ and write $u_c=\sqrt{1-\eta_c}e^{i\theta_c}$. Since $u_{c}$ does not vanish, taking $\varepsilon>0$ small enough we deduce that for any $\Psi\in\mathcal{X}(\R)$ satisfying $d_\mathcal{X}(\Psi,u_c)\leq\varepsilon$, we have $|||\Psi|-|u_{c}|||_{ L^\infty(\R)}\leq\inf_{x\in\R}|u_c(x)|/2$ so that 
		\begin{equation}\label{eq:closinf}
			\inf_{x\in\R}|\Psi(x)|\geq \inf_{x\in\R}|u_c(x)|/2.
		\end{equation} We can express $\partial_x\Psi$ in terms of $\eta$ and $\theta$,
		\begin{equation}\label{eq:Psiprim}
			\partial_x\Psi=e^{i\theta}\left(\frac{\eta_x}{2\sqrt{1-\eta}}+i\sqrt{1-\eta}\theta_x\right),
		\end{equation} and similarly for $u_c$.
		We also need to bound $||\Psi||_{ L^\infty(\R)}$ in terms of $\varepsilon$, $||u_c||_{ L^\infty(\R)}$ and $||\partial_xu_c||_{ L^2(\R)}$. To do so, notice that for any constants $y<x$, we have
		\begin{equation}\label{eq:morrey}
			|\Psi(x)|\leq|\Psi(y)|+\int_{y}^x|\partial_x\Psi|\leq|u_c(y)|+||\Psi(y)|-|u_c(y)||+\sqrt{x-y}||\partial_x\Psi||_{ L^2(\R)},
		\end{equation} where we used Cauchy--Schwarz inequality to bound the integral. Taking $0<x-y<1$ then integrating \eqref{eq:morrey} with respect to $y$ from $x-1$ to $x$ we get
		$|\Psi(x)|\leq |u_c|+\int_{\R}||u_c|-|\Psi||+||\partial_x\Psi||_{ L^2(\R)},$
		so that
		\begin{equation}\label{eq:inftybound}
			||\Psi||_{ L^\infty(\R)}\leq K_0(u_c,\varepsilon).
		\end{equation} 
		We show that $d_{\mathrm{hy}}(\Psi,u_{c})\leq Kd_\mathcal{X}(\Psi,u_c).$ 
		First notice that $|\Psi|^2\theta_x=\Im(\Psi_x\bar{\Psi})$, thus
		$$\partial_x\theta-\partial_x\theta_c=\Im(\partial_xu_c\bar{u_c})\left(\frac{1}{|\Psi|^2}-\frac{1}{|u_c|^2}\right)+\frac{1}{|\Psi|^2}(\Im((\partial_x\Psi-\partial_xu_c)\bar{\Psi})-\Im(\partial_xu_c(\bar{u_c}-\bar{\Psi}))),$$
		
		\begin{flalign}\label{eq:boundthetax}
			\text{hence, }\hspace{1.3cm} ||\partial_x\theta-\partial_x\theta_c||_{ L^2(\R)}\leq& \frac{\Im(\partial_xu_c\bar{u_c})(||\Psi||_{ L^\infty(\R)}+||u_c||_{ L^\infty(\R)})}{\inf_{\R}|\Psi|^2\inf_\R|u_c|^2}|||\Psi|-|u_c|||_{ L^2(\R)}&&\notag\\
			&+\frac{1}{\inf_\R|\Psi|^2}\left(||\Psi||_{ L^\infty(\R)}||\partial_x\Psi-\partial_xu||_{ L^2(\R)}+||\partial_xu_c(\bar{u_c}-\bar{\Psi})||_{ L^2(\R)}\right).
		\end{flalign}
		Since $\Psi(x)-u_c(x)=\Psi(0)-u_c(0)+\int_{0}^x\partial_x\Psi-\partial_xu_c$ for all $x\in\R$, using the Cauchy--Schwarz inequality, we get
		\begin{equation}
			|\partial_xu_c(x)||\Psi(x)-u_c(x)|\leq|\partial_xu(x)||\Psi(0)-u_c(0)|+|\partial_{x}u_c(x)\sqrt{|x|}|||\partial_x\Psi-\partial_xu_c||_{ L^2(\R)},\quad\text{ for all }x\in\R.
		\end{equation}
		From the decay estimate~\eqref{eq:decaydark} satisfied by $\partial_xu_c$, we can bound the last term in \eqref{eq:boundthetax}, therefore using also \eqref{eq:closinf} and \eqref{eq:inftybound}, we deduce that there exists $K_1=K_1(\varepsilon,u_c)$ such that 
		\begin{equation}\label{eq:equivtheta}
			||\partial_x\theta-\partial_x\theta_c||\leq K_1d_{\mathrm{hy}}(\Psi,u_c).	
		\end{equation}
		Similarly, we have $2\Re(\Psi_x\bar\Psi)=\eta_x$, thus 
		\begin{align}
			||\partial_x\eta-\partial_x\eta_c||_{ L^2(\R)}&\leq2||\Psi||_{ L^\infty(\R)}||\partial_x\Psi-\partial_xu_c||_{ L^2(\R)}+2||\partial_x{u_c}(\bar{u_c}-\bar\Psi)||_{ L^2(\R)},\notag\\
			&\leq2K_0d_{\mathrm{hy}}(\Psi,u_c)+2||\partial_xu_c||_{ L^2(\R)}||\Psi(0)|-|u_c(0)||+||\sqrt{x}\partial_xu_c||_{ L^2(\R)}|| \partial_x\Psi-\partial_xu_c||_{ L^2(\R)}\notag,\\
			&\leq K_2 d_{\mathrm{hy}}(\Psi,u_c).
		\end{align}
		Then, we also have $||\eta-\eta_c||_{ L^2(\R)}=||(|\Psi|+|u_c|)(|\Psi|-|u_c|)||_{ L^2(\R)}\leq(||\Psi||_{ L^\infty(\R)}+||u_c||_{ L^\infty(\R)})|||\Psi|-|u_c|||_{ L^2(\R)}$.
		We bound the last term $|\mathrm{Arg}(\Psi(0)/u_c(0))|$. Letting $\mathrm{Log}$ be the principal determination of the logarithm defined in $\C\backslash(-\infty,0]$ and taking $\varepsilon>0$ small enough, we get $$|\mathrm{Arg}(\Psi(0)/u_c(0))|=|\Im (\mathrm{Log}(\Psi(0)/u_c(0)))|\leq \sup_{|z-1|\leq\varepsilon/u_c(0)}|1/z|\left|\frac{\Psi(0)-u_c(0))}{u_c(0)}\right|,$$ which ends to prove that $d_{\mathrm{hy}}(\Psi,u_c)\leq K d_{\mathcal{X}}(\Psi,u_c)$.
		Ideas along the same lines as in Lemma~10 in \cite{chiron-stability} show that $d_{\mathcal{X}}(\Psi,u_c)\leq K d_{\mathrm{hy}}(\Psi,u_c)$.
	\end{proof}
	To prove Corollary~\ref{cor:stab}, we shall recall some properties of the function $P_\kappa(c)= P(u_{c,\kappa}),\text{ for all }c\in(0,\sqrt2).$ 
	\begin{lemma}[Proposition~5.6 and Lemma~5.7 in \cite{delaire2023exotic}]\label{lem:varp}
		There exists $\kappa_0<0$ with $\kappa_0\approx-3.636$ such that the two following statements hold.
		\begin{enumerate}
			\item If $\kappa \in[\kappa_0,1/2)$, then $P_\kappa'<0$ in $(0,\sqrt{2})$.
			\item If $\kappa<\kappa_0$, then there exists $\tilde{c}_\kappa$ such that $P'_\kappa>0$ in $(0,\tilde{c}_\kappa)$ while $P'_\kappa<0$ in $(\tilde{c}_\kappa,\sqrt{2})$.
		\end{enumerate}
	\end{lemma}
	\appendix
	\section{Extensions of the stability result}\label{appendix}
	Finally, we briefly describe the Lyapunov functional approach to the stability of the dark solitons (see \cite{chiron-stability}). 
	For the sake of conciseness, we chose not to use this method in this work to avoid tedious computations.
	As pointed out by C.~A.~Stuart in Section~4--5 in \cite{stuartstab}, in the Hamiltonian system framework of \cite{grillakisshatah}, the stability criterion \eqref{eq:critstab} together with \textit{Assumptions~1--3} allow us to construct a Lyapunov function invariant by translation by taking $K$ large enough in 
	\begin{equation}\label{eq:lyap}
		V(\eta,v)=E(\eta,v)-cP(\eta,v)-E(u_{c})+cP(u_{c})+K(P(\eta,v)-P(u_{c}))^2,
	\end{equation} giving a quantitative measure of stability. The heuristic is the following, since $V$ is constant on the flow of solutions, to ensure the stability, we want $V$ to satisfy, for some $\alpha>0$ $$V(\eta,v)\geq\alpha ||(\eta,v)-(\eta_{c},v_{c}) ||^2_{X},$$ for some $\varepsilon>0$ and all $\eta$, $v$ satisfying $\inf_{s\in\R}||(\eta,v)-(\eta_{c},v_{c})(\cdot-s) ||_{X}<\varepsilon$. This problem reduces to the positive definiteness of $V''(\eta_{c},v_{c})$ which writes
	$$\langle V''(\eta_{c},v_{c})(\delta\eta,\delta v),(\delta\eta,\delta v)\rangle=\langle H_c(\delta\eta,\delta v),(\delta\eta,\delta v)\rangle+2K(\langle P'(u_{c}),(\delta\eta,\delta v)\rangle )^2 \geq C||(\delta\eta,\delta v) ||^2_{X},$$
	for some $C>0$.
	We readily compute $-P'(u_{c})=\partial_c(E'-cP')(u_{c})$ and since $E'(u_c)-cP'(u_{c})=0$ for every $c\in(0,\sqrt{2})$, we get $\partial_c(E'(u_c)-cP'(u_{c}))=H_c\partial_cu_{c}-P'(u_{c})=0$ by the chain rule. Testing this linear form against $\partial_c(\eta_c,v_c)\in X$, if \eqref{eq:critstab} holds, we deduce that,
	\begin{equation}\label{eq:dcu}
		\langle H_c\partial_c(\eta_c,v_c),\partial_c(\eta_c,v_c)\rangle=\frac{d}{dc}P(\eta_c,v_c)<0.
	\end{equation}
	By Lemma~5.3 in \cite{stuartstab}, considering \textit{Assumption~3}, we can assume the decomposition to be the orthogonal decomposition  
	\begin{equation}
		u=a\partial_c(\eta_c,v_c)+b\tau'(0)(\eta_c,v_c)+p, \quad\text{ for all }u\in X,
	\end{equation}
	where $P=\{\partial_c(\eta_c,v_c), \tau'(0)(\eta_c,v_c)\}^\perp$. Then we can show  that taking $K$ big enough in \eqref{eq:lyap} yield the coercivity of the quadratic form $V''$ on $\mathrm{span}\{\partial_c(\eta_c,v_c)\}$ which upon further computations leads to (see Lemma~4.3 in \cite{stuartstab}) 
	\begin{equation}
		\langle V''(\eta_c,v_c)(\delta\eta\delta v),(\delta\eta,\delta v )\rangle\geq C||\Pi(\delta\eta,\delta v)||_X^2,
	\end{equation}
	for some $C>0$ and where $\Pi$ is the orthogonal projection on $\{\tau'(0)(\eta_c,v_c)\}^\perp$.
	Finally if $(\delta\eta,\delta v)=(\eta-\eta_c,v-v_c)$ then we can show that (see Lemma~2.1 in \cite{stuartstab})
	\begin{equation}\label{eq:proj}
		||\Pi(\delta\eta,\delta v)||=\inf_{s\in\R}||(\eta,v)-(\eta_c,v_c)(\cdot-s)||_X.
	\end{equation} For any solution $(\eta,v)$ to \eqref{eq:hydrot}, one can readily bound from above $V((\eta,v))(t)$ for all $t$ by the quantity \eqref{eq:proj} at $t=0$. Therefore a solution starting close to the orbit of traveling waves cannot leave it, ensuring stability.
		This approach is implemented for the cubic-quintic nonlinear Schr\"odinger equation with nonzero conditions at infinity by I.~V.~Barashenkov in~\cite{barashenkovcrit}.

		\begin{merci}
			The author is grateful to Andr\'e de Laire, the reviewer and editors for helpful discussions and their interest in the publication of this paper. 
			The authors acknowledge support from the Labex CEMPI (ANR-11-LABX-0007-01). E.~Le~Quiniou was also supported by the R\'egion Hauts-de-France.
		\end{merci}

		\bibliographystyle{abbrv}   
 	
	\end{document}